\documentclass[11pt]{amsproc}
\usepackage{amstext,amsmath,amssymb,amsthm, esint}
\usepackage{latexsym}
\usepackage{exscale}
\usepackage{color}

\RequirePackage[cp1251]{inputenc}
\RequirePackage{srcltx}

\numberwithin{equation}{section}
\RequirePackage{srcltx}

\RequirePackage[cp1251]{inputenc}

\numberwithin{equation}{section}

\renewcommand\Im{\operatorname{Im}}

\newcommand\ddfrac[2]{\frac{\displaystyle #1}{\displaystyle #2}}

\newtheorem{Thm}{Theorem}[section]
\newtheorem{Lemma}[Thm]{Lemma}

\theoremstyle{remark}

\begin{document}

\title[PDEs and hypercomplex analytic functions]{PDEs and hypercomplex analytic functions}
\author{David Harper}
\address{Florida International University,
Department of Mathematics,
Miami, FL 33199, USA}
\email{dharp012@fiu.edu}
\subjclass[2016]{}
 \begin{abstract}
Hypercomplex numbers are unital algebras over $\mathbb R$. We offer a short demonstration of the practical value of  hypercomplex analytic functions in the field of partial differential equations.

\end{abstract} 
 
\maketitle

\section{Introduction}
We show that if complex analytic functions and the complex numbers themselves had not been discovered before the study of the two-dimensional Laplace equation, $\Delta h =  h_{xx} + h_{yy} = 0$, then they would have likely been discovered as a result of it. In this paper, a hypercomplex analytic function is an element of a class of differentiable functions, $C^1 (A)$, over a two-dimensional commutative unital algebra, $A$. Although, hypercomplex analytic functions have appeared in the literature, e.g. \cite{Ketchum},  \cite{Ward2}, we will explore the connections between these atypical theories of differentiable functions and certain partial differential equations. We define $\Gamma_{\alpha,\beta}$ as the partial differential operator,
\begin{center}
$\Gamma_{\alpha,\beta} = \partial_{x}^2 + \alpha \partial_{x}\partial_{y} + \beta \partial_{y}^2$ with $\alpha,\beta \in \mathbb R$.
\end{center} 
We have selected $\Gamma_{\alpha,\beta}$ in this manner as it is a generalization of the harmonic operator $\Delta = \partial_x^2 + \partial_y^2$, in the special case, $\Gamma_{0,1}$. Now, we introduce and utilize these new classes of differentiable functions to answer the following important questions;

\begin{itemize}

\item[1.]	What kind of continuous change of variables preserve solutions to the problem,
\begin{equation} \label{Eqn1}
 \Gamma_{\alpha,\beta}h = 0 
\end{equation}
for $h=h(x,y)$ and $(x,y) \in \Omega \subseteq \mathbb R^2$?

\item[2.]	What is a general form of solution to the generalized Bi-harmonic Equation,

\begin{equation} \label{Eqn2}
 \Gamma_{\alpha,\beta}^2 h = 0 
 \end{equation} 
 for $h=h(x,y)$ and $(x,y) \in \Omega \subseteq \mathbb R^2$ \label{Eqn2}

\end{itemize}

In fact, we have chosen to select these questions as they are naturally answered in the case of the harmonic operator via the use of complex analytic functions, thus, rendering the results obtained in the case $\Gamma_{\alpha,\beta}$ true generalizations.  

We define the commutative algebra, 
\begin{center}
$A_{ \lambda_1 , \lambda_2 } = \{ t + sj | (t,s) \in \mathbb R^2 \}$ where, 
$ j^2 = - \lambda_2  - \lambda_1 j $ and $j \not\in \mathbb R$. 

\end{center}

The numbers $z = t + sj \in A_{ \lambda_1 , \lambda_2 }$ will have a {\it real} part, $Re_{ \lambda_1 , \lambda_2 }(z) = t$, and an {\it imaginary } part $Im_{ \lambda_1 , \lambda_2 }(z) = s$.

The conjugate will also be defined in the analogous manner, 
\begin{center}
$\bar{z}^{ \lambda_1 , \lambda_2 } = t - sj$.
\end{center}

Over these algebras we may define functions,
\begin{center}
$ F : A_{ \lambda_1 , \lambda_2 }  \rightarrow A_{ \lambda_1 , \lambda_2 } $, \\
$ F(x,y) = u(x,y) + jv(x,y), $ 
\\

\end{center}
where $u,v$ are functions from $\mathbb{R}^2$ into $\mathbb{R}$.
\\
Such functions will be called $ A_{ \lambda_1 , \lambda_2 }$-differentiable if they satisfy the generalized Cauchy-Riemann conditions,
\begin{center} 
$ u_x = v_y + \lambda_1 v_x, $ \\ 
$u_y = - \lambda_2 v_x. $  
\end{center}

See e.g. \cite{Ward1} for related work on generalized Cauchy-Riemann conditions.
If the reader replaces $\lambda_1 = 0 , \lambda_2 = 1$, they will obtain the system of complex numbers and complex analytic functions.  

\subsection{Main Results}
We will now list the main results which serve as answers to our primary questions. 
The following theorem answers question $1$. 

\begin{Thm}  \label{thm1}
Let $\Omega$ be a domain in $\mathbb R^2$ and let $h: \Omega \rightarrow \mathbb R$ be a solution to,
\begin{center}
$ \Gamma_{\alpha,\beta}h = 0 $
\end{center}

for all $(x,y) \in \Omega$ and let $F = (u,v): \Omega \rightarrow \mathbb R^2 $. Then, the composition $H = h \circ F$ is also a solution to \ref{Eqn1} if, $F = u + jv $ defines an $ A_{ \frac{\alpha}{\beta},\frac{1}{\beta} }$-differentiable function on the region $\Omega$. 
\end{Thm}

The next theorem answers question 2. 
\begin{Thm} \label{thm2}
Let $f = u + jv$, and $g = k + jw$ be $ A_{ \frac{\alpha}{\beta},\frac{1}{\beta} }$-differentiable functions on a domain $\Omega $ in  $\mathbb R^2$. Then, $h(x,y) = \Im_{ \lambda_1 , \lambda_2 }( \bar{z}^{ \lambda_1 , \lambda_2 }f(z) + g(z) ),$  is a solution to the equation, 
\begin{center}
$ \Gamma_{\alpha,\beta}^2 h  = 0 $ 
\end{center}
on the region, $\Omega$.

\end{Thm}

Again, if the reader considers the complex-harmonic case, $\lambda_1 = 0 , \lambda_2 = 1$, then they will see familiar theorems, \cite{Charles},\cite{{Rudin}},\cite{Kre}. 
\\
The following Lemma is very useful and as well constitutes an informative generalization of complex analytic functions.
\begin{Lemma} \label{lemma1}
Let $F = u + j v$ be an $ A_{ \frac{\alpha}{\beta},\frac{1}{\beta} }$-differentiable function.
\\ Then, 
\begin{center}
$\Gamma_{\alpha,\beta}u = \Gamma_{\alpha,\beta}v = 0$.
\end{center} 
\end{Lemma}

\begin{proof} \ref{lemma1}
Consider the generalized Cauchy-Riemann condition. 
\begin{center} 
$ u_x = v_y + \ddfrac{ \alpha }{ \beta } v_x $  and $u_y = - \ddfrac{ 1 }{ \beta } v_x $  
\end{center}
and so,
\begin{center}
$ v_{yy} + \ddfrac{ \alpha }{ \beta } v_{xy} = - \ddfrac{ 1 }{ \beta } v_{xx}$
\end{center}
gives,  
\begin{center}
$  \Gamma_{\alpha,\beta}v = v_{xx}+ \alpha v_{xy} + \beta v_{yy} = 0$.
\end{center}
Then, similarly, we infer,  $\Gamma_{\alpha,\beta}u = 0$

\end{proof}

\section{Proofs of presented results}

\begin{proof}[ Proof \ref{thm1}] 
Begin by calculating the derivatives and grouping terms.

$(u_{xx}+\alpha u_{xy} + \beta u_{yy} )h_u + (u_x^2+\alpha u_x u_y + \beta u_y^2)h_{uu} + (2u_xv_x+\alpha(u_xv_y+u_yv_x) +  2\beta u_y v_y )h_{uv} + (v_x^2 + \alpha v_x v_y + \beta v_y^2 )h_{vv} + (v_{xx} + \alpha v_{xy} + \beta v_{yy} )h_v  = **$
Then, recall by Lemma \ref{lemma1}, both $u$ and $v$ are solutions to the equation \ref{Eqn1}. Then we begin considering each factor of $h_{uu},h_{uv},h_{vv}$. We apply different variations of the Generalized Cauchy-Riemann conditions to see that the terms all simplify nicely.
\begin{itemize}

\item $u_x^2+\alpha u_x u_y + \beta u_y^2 = u_x(u_x+\alpha u_y)-u_y(-\beta u_y) = | J(F) | $

\item $ 2 u_x v_x+\alpha(u_x v_y+u_y v_x) +  2 \beta u_y v_y = \alpha \big ( \ddfrac{2}{ \alpha }(v_y - \alpha u_y )v_x + u_x v_y +  v_x u_y + \ddfrac{ 2 \beta }{ \alpha } ( \ddfrac{-1}{ \beta }v_x) v_y  \big ) = \alpha | J(F) |   $

\item $v_x^2+\alpha v_x v_y + \beta v_y^2 = \beta \big( - ( -\ddfrac{1}{\beta} v_x ) v_x +  ( \ddfrac{\alpha}{\beta}v_x + v_y )v_y \big ) = \beta |J(F) | $

\end{itemize}
Then so, 

\begin{center}
$ ** = | J(F) |  h_{uu} + \alpha | J(F) | h_{uv} +\beta |J(F) | h_{vv} = |J(F) | \big (  h_{uu} + \alpha h_{uv} +\beta h_{vv} \big ) = 0 $
\end{center}

Thus completing the proof.
\end{proof}


\begin{proof}[ Proof \ref{thm2}] 
Begin by calculating derivatives of $h(x,y) = \Im_{ \lambda_1 , \lambda_2 }( \bar{z}^{ \lambda_1 , \lambda_2 }f(z) + g(z) )$ and grouping terms.
If we denote $z = x+jy$ with $x,y \in \mathbb{R}$, then we first group terms by factors of $y$,
\begin{center}
$\ddfrac{ \alpha }{ \beta }  v^{(4,0)} -  u^{(4,0)} + 2  \ddfrac{ \alpha^2 }{ \beta }  v^{(3,1)}  - 2 \alpha  u^{(3,1)} + (\alpha^2 + 2 \beta)\ddfrac{ \alpha }{ \beta }  v^{(2,2)}  -  (\alpha^2 + 2 \beta) u^{(2,2)} + 2 \alpha^2 v^{(1,3)} -  2 \alpha \beta  u^{(1,3)} + \alpha \beta   v^{(0,4)}  - \beta^2  u^{(0,4)} = \ddfrac{ \alpha }{ \beta } \big (   v^{(4,0)} + 2  \alpha  v^{(3,1)} + (\alpha^2 + 2 \beta) v^{(2,2)}+ 2 \alpha \beta v^{(1,3)}+ \beta ^2  v^{(0,4)} \big )- \big (u^{(4,0)} + 2 \alpha  u^{(3,1)} +  (\alpha^2 + 2 \beta) u^{(2,2)} + 2 \alpha \beta  u^{(1,3)} + \beta^2  u^{(0,4)} \big ) = \ddfrac{ \alpha }{ \beta } \Gamma^2_{ \alpha , \beta } v -  \Gamma^2_{ \alpha , \beta } u = 0 $
\end{center}
In terms of $x$,
\begin{center}
$ v^{(4,0)} + 2 \alpha v^{(3,1)} + (\alpha^2 + 2 \beta )v^{(2,2)} + 2 \alpha \beta v^{(1,3)} + \beta^2 v^{(0,4)} = \Gamma^2_{ \alpha , \beta } v  = 0 $
\end{center}
Then, the function $w$ and its derivatives,
\begin{center}
$ w^{(4,0)} + 2 \alpha w^{(3,1)} + (\alpha^2 + 2 \beta )w^{(2,2)} + 2 \alpha \beta w^{(1,3)} + \beta^2 w^{(0,4)} = \Gamma^2_{ \alpha , \beta } v = 0 $
\end{center}
Finally, the functions $u,v$ and their derivatives with constant factors,
\begin{center}
$ 4v^{(3,0)} + 2 \ddfrac{ \alpha^2 }{ \beta }  v^{(3,0)} - 2 \alpha  u^{(3,0)} + 6 \alpha v^{(2,1)} + 2 \ddfrac{ \alpha^2 }{ \beta } (\alpha^2 + 2\beta) v^{(2,1)} - 2(\alpha^2 + 2\beta) u^{(2,1)} + 2(\alpha^2 + 2\beta) v^{(1,2)} + 6 \alpha^2 v^{(1,2)} - 6 \alpha \beta u^{(1,2)} + 2 \alpha \beta  v^{(0,3)} - 4 \beta^2 u^{(0,3)} + 4 \alpha \beta v^{(0,3)}.$

\medskip

\end{center}

Next we differentiate the Generalized Cauchy Riemann Conditions and substitute $u$,
\begin{center}
$ u^{(1,0)} = v^{(0,1)} +  \ddfrac{ \alpha }{ \beta } v^{(1,0)} $ 

$ u^{(0,1)} = - \ddfrac{ 1 }{ \beta } v^{(1,0)} $ 

\end{center}

\begin{center}
$ 4v^{(3,0)} + 2 \ddfrac{ \alpha^2 }{ \beta }  v^{(3,0)} - 2 \alpha  \big(  v^{(2,1)} + \ddfrac{ \alpha }{ \beta } v^{(3,0)}  \big )  + 6 \alpha v^{(2,1)}+ 2(\alpha^2 + 2\beta) \big( v^{(1,2)}  + \ddfrac{ \alpha }{ \beta } v^{(2,1)} \big ) - 2(\alpha^2 + 2\beta) u^{(2,1)} + 6 \alpha^2 v^{(1,2)} - 6 \alpha \beta \big ( v^{(0,3)} + \ddfrac{ \alpha }{ \beta }v^{(1,2)} \big )+ 2 \alpha \beta  v^{(0,3)} + 4 \beta^2 \big( \ddfrac{ 1 }{ \beta }  v^{(1,2)} \big ) + 4 \alpha \beta v^{(0,3)} = 4v^{(3,0)} + 4 \alpha v^{(2,1)} + 4 \beta  v^{(1,2)}  = 4 \partial_x  \big ( v^{(2,0)} +  \alpha v^{(1,1)} +  \beta  v^{(0,2)} \big ) = 4 \partial_x \Gamma_{ \alpha , \beta } v =  0$  .

\end{center}

This completes the proof. 
\end{proof}

\section{Concluding Remarks}
We hope that this paper has helped to give perspective on the usefulness of the hyper-complex number system and the corresponding differentiable functions. As well, to put perspective on the place of the complex-analytic functions within the mathematical landscape relative to their, often equally significant, counterparts.


\begin{thebibliography}{1}
  
    
  \bibitem{Charles} L. Charles, {\em On generation of solutions of the biharmonic equation in the plane by conformal mappings}, Pacific J. Math. 3 (1953), no. 2, 417-436.
  
  \bibitem{Ketchum} P. W. Ketchum, {\em Analytic functions of hypercomplex variables}, Trans. Amer. Math. Soc. 30 (1928), no. 4, 641-667. 

  \bibitem{Kre} E. Kreyszig {\em Advanced Engineering Mathematics} (1972), John Wiley and Sons, Inc.
  

   \bibitem{Rudin} W. Rudin, {\em Real and Complex Analysis, 3rd ed.} (1987), McGraw-Hill, Inc.  
  
 \bibitem{Ward1} J. A. Ward, {\em 
From Generalized Cauchy-Riemann Equations to Linear Algebras}, Proceedings of the American Mathematical Society
Vol. 4, No. 3 (Jun., 1953), 456-461    
  
  \bibitem{Ward2} J. A. Ward, {\em A theory of analytic functions in linear associative algebras}, Duke Math. J. 7 (1940), 233-248. 
  


  



  \end{thebibliography}
\end{document}